\documentclass[onefignum,onetabnum]{siamart220329}

\ifpdf
\DeclareGraphicsExtensions{.eps,.pdf,.png,.jpg}
\else
  \DeclareGraphicsExtensions{.eps}
\fi

\usepackage{algorithm,algorithmicx,algpseudocode}
\usepackage{appendix,tabularx}
\usepackage{amssymb,amsfonts,amsmath,enumitem}
\usepackage[mathscr]{eucal}
\usepackage{bm}
\usepackage{xcolor}  
\usepackage{xparse}
\usepackage{extarrows} 
\usepackage{booktabs} 
\usepackage{siunitx} 
\usepackage{tablefootnote} 
\usepackage{tikz}
\usetikzlibrary {quotes} 
\usetikzlibrary {arrows.meta} 
\usepackage{hyperref}
\usepackage[nameinlink]{cleveref} 
\usepackage{amsopn}
\usepackage{multirow}
\usepackage{makecell}
\usepackage{amsbsy}
\usepackage{stmaryrd}
\usepackage{braket}  
\usepackage{xspace}
\usepackage{microtype}
\usepackage{algorithmicx}  
\usepackage {graphicx,epstopdf} 
\usepackage {array}
\usepackage[caption=false]{subfig}

\NewDocumentCommand \tensor {O{}m} {\boldsymbol{#1\mathscr{\MakeUppercase{#2}}}} 

\newcommand{\bb}[1]{\mathbb{#1}}

\newcommand{\LineRef}[1]{\hyperref[#1]{Line~\ref{#1}}}

\newsiamremark{remark}{Remark}
\newsiamremark{hypothesis}{Hypothesis}
\crefname{hypothesis}{Hypothesis}{Hypotheses}
\newsiamthm{claim}{Claim}

\headers{Randomized dual singular value decomposition}{Mengyu Wang, Jingchun Zhou, and Hanyu Li}

\title{Randomized dual singular value decomposition}

\author{Mengyu Wang\thanks{College of Mathematics and Statistics, Chongqing University, Chongqing 401331, P.R. China
  (\email{995647966@qq.com;\ zhoujc@cqu.edu.cn;\ lihy.hy@gmail.com or hyli@cqu.edu.cn}).}
\and Jingchun Zhou\footnotemark[1] \and Hanyu Li\footnotemark[1]}

\begin{document}
\begin{sloppypar}
\maketitle

\begin{abstract}
We first propose a concise singular value decomposition of dual matrices. Then, the randomized version of the decomposition is presented. It can significantly reduce the computational cost while maintaining the similar accuracy. We analyze the theoretical properties and illuminate the numerical performance of the randomized algorithm. 
\end{abstract}



\section{Introduction}
\label{sec:int}
Dual numbers were first introduced by \cite{clifford1871preliminary} as part of his extensive research on Clifford algebras. They take the form $q = q_s + q_i \epsilon$, where $q_s$ and $q_i$ are real or complex numbers, and $\epsilon$ is an infinitesimal unit satisfying $\epsilon^2 = 0$.  
Accordingly, we call $\mathbf{A} = \mathbf{A}_s + \mathbf{A}_i\epsilon$ with $\mathbf{A}_s, \mathbf{A}_i \in \mathbb{C}^{m \times n}$ a dual complex matrix. The set of this kind of matrices 
is denoted as 
$\mathbb{D}\mathbb{C}^{m \times n}$. 
Dual numbers and dual matrices have found applications in various fields such as 
robotics \cite{pradeep1989use,gu1987dual}, classical mechanics \cite{angeles1998application,pennestri2009linear}, and computer vision \cite{penunuri2019dual}. Moreover, they are also linked with screw theory \cite{fischer1998dual}. 
For decompositions of dual matrices, 
some authors discussed singular value decomposition (SVD) \cite{qi2022low,wei2024singular}, QR decomposition \cite{xu2024qr,pennestri2007linear}, UTV decomposition \cite{xu2024utv} and rank decomposition \cite{wang2023dual}. In this paper, we foucus on SVD of dual matrices.

As mentioned above, several scholars have contributed to the development and application of SVD for dual matrices. Specifically, Pennestrì et al. \cite{pennestri2018moore} obtained the SVD of dual real full column rank matrices by solving the linear systems involving Kronecker products.  Gutin \cite{gutin2022generalizations} developed two other SVDs applicable to square dual real matrices.  Qi et al. \cite{qi2022low} demonstrated that the dual complex matrix can be reduced to a diagonal non-negative dual matrix by unitary transformation. 
Wei et al. \cite{wei2024singular} introduced a compact dual SVD (CDSVD), which has more advantages compared with the above ones. 
However, CDSVD configures the singular value matrix as a dual matrix, which results in a complicated algorithm. For clarity, we simplify the matrix as a real matrix and propose a concise CDSVD (CCDSVD).

Considering that SVD faces challenges when dealing with large-scale matrices and randomized methods can significantly reduce the computational cost, we further present the randomized version of CCDSVD. 
Note that randomized SVD of real or complex matrices has been investigated extensively; see e.g., \cite{halko2011finding,gu2015subspace,martinsson2016randomized,yu2018efficient}. 

The rest of this paper is organized as follows. In Section 2, we propose the CCDSVD. Its randomized version is given in Section 3. In Section 4, we present
the error analysis. Section 5 is devoted to experimental
results. Finally, the whole paper is concluded with some remarks. 

\section{CCDSVD}
\label{sec:scdsvd}
With the technique in \cite{xu2024utv} for UTV decomposition of dual matrices, 
we can obtain CCDSVD in the following theorem. 

\begin{theorem}
    Let \(\mathbf{A} = {\bf A}_s + {\bf A}_i\epsilon \in {\bb {DC}}^{m\times n}\) with $ m\ge n $. Assume that ${\bf A}_s = \mathbf{U}_s {\bf \Sigma}_s \mathbf{V}_s^*$ is a compact SVD of ${\bf A}_s$. Then the CCDSVD of $\bf A$ exists if and only if
    \begin{equation*}
        (\mathbf{I}_m - \mathbf{U}_s \mathbf{U}_s^*) \mathbf{A}_i (\mathbf{I}_n - \mathbf{V}_s \mathbf{V}_s^*) = \mathbf{O}_{m \times n}.
    \end{equation*}
    Furthermore, if ${\bf A}$ has a CCDSVD, then there exists a particular pair $({\bf U}_s, {\bf V}_s)$ such that ${\bf A} = {\bf U} {\bf \Sigma} {\bf V}^*$, where ${\bf U} = {\bf U}_s + {\bf U}_i\epsilon \in {\bb {DC}}^{m\times r}$ and ${\bf V} = {\bf V}_s + {\bf V}_i\epsilon \in {\bb {DC}}^{n\times r}$ both have unitary columns and ${\bf \Sigma} = {\bf \Sigma}_s \in {\bb {R}}^{r\times r}$ is diagonal positive. Moreover, the expressions of ${\bf U}_i$ and ${\bf V}_i$ are given by
    \begin{align*}
        & {{\bf U}_i} = \left( {{{\bf I}_m} - {{\bf U}_s}{\bf U}_s^*} \right){{\bf A}_i}{{\bf V}_s}{{\bf \Sigma}_s^{ - 1}} + {{\bf U}_s}{\bf P}, \\
        & {{\bf V}_i} ={\bf A}_i^*{{\bf U}_s}{\left( {{\bf \Sigma}_s^{ - 1}} \right)^*} - {{\bf V}_s}{\bf \Sigma}_s^*{{\bf P}^*}{\left( {{\bf \Sigma}_s^{ - 1}} \right)^*},
    \end{align*}
    where ${\bf P} \in {\bb C}^{r\times r}$ is an arbitrary skew-Hermitian matrix.
\end{theorem}

The specific algorithm is summarized in \Cref{alg:dsvd2}. 

\begin{algorithm}[htbp]
\caption{CCDSVD}
\label{alg:dsvd2}
\textbf{Input:} \(\mathbf{A} = \mathbf{A}_s + \mathbf{A}_i \epsilon \in \mathbb{D}\mathbb{C}^{m \times n}\) with \(m \geq n\).

\textbf{Output:} ${\bf U} = {\bf U}_s + {\bf U}_i\epsilon \in {\bb {DC}}^{m\times r}$, ${\bf V} = {\bf V}_s + {\bf V}_i\epsilon \in {\bb {DC}}^{n\times r}$ and ${\bf \Sigma} = {\bf \Sigma}_s \in {\bb {R}}^{r\times r}$.
\begin{algorithmic}[1]
    \State Decompose \(\mathbf{A}_s\) by a compact SVD: \(\mathbf{A}_s = \mathbf{U}_s {\bf \Sigma}_s \mathbf{V}_s^*\), where \(\mathbf{U}_s \in \mathbb{C}^{m \times r}\) and \(\mathbf{V}_s \in \mathbb{C}^{n \times r}\) both have unitary columns, and \({\bf \Sigma}_s  \in {\bb {R}}^{r\times r} \) is diagonal positive.

    \While {the condition \((\mathbf{I}_m - \mathbf{U}_s \mathbf{U}_s^*) \mathbf{A}_i (\mathbf{I}_n - \mathbf{V}_s \mathbf{V}_s^*) = \mathbf{O}_{m \times n}\) is satisfied }
    \State Set ${\bf P}$ to be a skew-Hermitian matrix.
    \State Set ${{\bf U}_i} = \left( {{{\bf I}_m} - {{\bf U}_s}{\bf U}_s^*} \right){{\bf A}_i}{{\bf V}_s}{{\bf \Sigma}_s^{ - 1}} + {{\bf U}_s}{\bf P}$.
    \State Set $ {{\bf V}_i} ={\bf A}_i^*{{\bf U}_s}{\left( {{\bf \Sigma}_s^{ - 1}} \right)^*} - {{\bf V}_s}{\bf \Sigma}_s^*{{\bf P}^*}{\left( {{\bf \Sigma}_s^{ - 1}} \right)^*}$.
    \EndWhile
\end{algorithmic}
\end{algorithm}

\begin{remark}
    As mentioned in \Cref{sec:int}, the key difference between CDSVD and CCDSVD is that ${\bf \Sigma}$ in the latter is a real matrix but not a dual one as in CDSVD. The change reduces the computation of the decomposition to some extent. Another motivation is that the standard part of ${\bf \Sigma}$ in CDSVD carries the main information; see 
    \cite{wei2024singular} for details.
\end{remark}

\section{Randomized CCDSVD}
\label{sec:rand}
 The specific algorithm is provided in \Cref{alg:rdsvd}.

\begin{algorithm}[htbp]
\caption{Randomized CCDSVD (RCCDSVD)}
\label{alg:rdsvd}
\textbf{Input:} ${\bf A} = {\bf A}_s + {\bf A}_i\epsilon \in {\bb {DC}}^{m\times n}$, target rank $r$, oversampling parameter $p$ and power scheme parameter $q$.

\textbf{Output:} ${\bf U} = {\bf U}_s + {\bf U}_i\epsilon \in {\bb {DC}}^{m\times r}$, ${\bf V} = {\bf V}_s + {\bf V}_i\epsilon \in {\bb {DC}}^{n\times r}$ and ${\bf \Sigma} = {\bf \Sigma}_s \in {\bb {R}}^{r\times r}$.
\begin{algorithmic}[1]
\State Set $\ell = r + p$, and draw an $n\times \ell$ random complex matrix ${\bf \Omega}$.
\State Construct ${\bf Y}_0 = {\bf A}{\bf \Omega}$ .
\For{$j = 1,\cdots ,q$}
    \State Form $\widetilde{\bf Y}_j = {\bf A}^*{\bf Y}_{j-1}$ and compute its dual QR factorization $\widetilde{\bf Y}_j = \widetilde{\bf Q}_j\widetilde{\bf R}_j$.
    \State Form ${\bf Y}_j = {\bf A}^*\widetilde{\bf Q}_{j}$ and compute its dual QR factorization $\widetilde{\bf Y}_j = {\bf Q}_j{\bf R}_j$.
\EndFor
    \State  Construct a dual matrix ${\bf Q}$ by  thin dual QR  of ${\bf Y}_q$.
    \State Form ${\bf B} = {\bf Q}^* {\bf A}$.
    \State Compute the CCDSVD of the small dual matrix: ${\bf B} = \overline{\bf U}{\bf \Sigma}{\bf V}^*$.
    \State Set ${\bf U} = {\bf Q}\overline{\bf U}$.
\end{algorithmic}
\end{algorithm}

\begin{remark}
The random matrix ${\bf \Omega}$ in \Cref{alg:rdsvd} is complex but not dual. The main purpose is to reduce the computation and to avoid information redundancy. This is because ${\bf A}{\bf \Omega} = {{\bf A}_s}{{\bf \Omega} _s} + \left( {{{\bf A}_i}{{\bf \Omega} _s} + {{\bf A}_s}{{\bf \Omega}_i}} \right)\epsilon$ if ${\bf \Omega} ={\bf \Omega}_s+{\bf \Omega}_i\in {\bb {DC}}^{n \times \ell}$. Moreover, 
Proposition 4.4 in \cite{wei2024singular} shows that 
the approximation to the standard part of ${\bf A}_s$ plays a key role in  the approximation to the 
original dual matrix ${\bf A}$. This implies that it is reasonable to mainly capture the action of ${\bf A}_s$ when devising  randomized algorithms. 
Besides, numerical experiments in \Cref{sec:exp} also demonstrate that the present setting performs better 
in most cases.
\end{remark}


\section{Error analysis}
\label{sec:err}
We first present the definition of Frobenius norm of dual matrices.
\begin{definition}[\cite{qi2022low}]
    Let ${\bf A} = {\bf A}_s + {\bf A}_i\epsilon \in {\bb {DC}}^{m\times n}$. The Frobenius norm of ${\bf A}$ is defined as  
    \begin{equation*}
        {\left\| {{{\bf A}}} \right\|_F} = \left\{ \begin{array}{l}
        {\left\| {{{\bf A}_s}} \right\|_F} + \frac{{\left\langle {{{\bf A}_s},{{\bf A}_i}} \right\rangle }}{{{{\left\| {{{\bf A}_s}} \right\|}_F}}}\epsilon \qquad {\rm {if}} \;\; {{\bf A}_s} \ne {{\bf O}_{m \times n}},\\
        {\left\| {{{\bf A}_i}} \right\|_F}\epsilon \;\;\;\;\; \quad \qquad \qquad {\rm {otherwise}},
    \end{array} \right.
    \end{equation*}
where the inner product of matrices is over the real number field, that is, $\left\langle {{{\bf A}_s},{{\bf A}_i}} \right\rangle  = {\mathop{\rm {Re}}\nolimits} \left( {{\rm trace}\left( {{\bf A}_s^*{{\bf A}_i}} \right)} \right)$.
\end{definition}

\begin{theorem}[Average Frobenius error]
\label{the:fe}
     The expected approximation error for the approximation ${\widehat {\bf A}}={\bf U}{\bf \Sigma}{\bf V}^* $ via RCCDSVD with $q=0$, $r \ge 2$ and  $p \ge 2$ satisfying $r + p \le {\rm {min}}\left\{ m, n \right\}$ satisfies
     \begin{equation*}
         {\bb E} {\left\| {\widehat {\bf{A}} - {\bf{A}}} \right\|_F} \le {\left( {1 + \frac{r}{{p - 1}}} \right)^{1/2}}{\left( {\sum\nolimits_{j > r} {\sigma _j^2} } \right)^{1/2}} + \sqrt{2\left\langle {{\bf E}_s, {\bf A}_i} \right\rangle \epsilon},
     \end{equation*}
     where ${\bf E}_s = {\bf A}_s - {\bf U}_s{\bf \Sigma}_s{\bf V}_s^*$.
     \end{theorem}
     \begin{proof} 
         From ${\bf A} - {\bf U}{\bf \Sigma}{\bf V}^* = {\bf E}_s+{\bf E}_i\epsilon$, we have 
         ${\bf E}_i = {\bf A}_i - {\bf U}_s{\bf \Sigma}_s{\bf V}_i^* - {\bf U}_i{\bf \Sigma}_s{\bf V}_s^*$ and
         \begin{equation*}
             {\left\| {\widehat {\bf{A}} - {\bf{A}}} \right\|_F^2} = {\left\| { {\bf E}_s + {\bf E}_i \epsilon} \right\|^2_F} = {\left\| { {\bf E}_s} \right\|^2_F + 2\left\langle {{\bf E}_s, {\bf E}_i} \right\rangle \epsilon},
         \end{equation*}
         which together with \cite[Theorem 10.5]{halko2011finding}, the result in the proof of \cite[Proposition 4.4]{wei2024singular}, and the Cauchy-Schwarz inequality implies the desired result. 
         \end{proof}

To interpret the role of the infinitesimal part in low-rank approximation, Wei et al. \cite{wei2024singular} defined a quasi-metric to measure the distance between two dual matrices.
\begin{definition}[\cite{wei2024singular}]
     Let ${\bf A} = {\bf A}_s + {\bf A}_i\epsilon, {\bf B} = {\bf B}_s + {\bf B}_i\epsilon \in {\bb {DC}}^{m\times n}$. The quasi-metric of the two dual matrices ${\bf A}$ and ${\bf B}$ is defined as   
    \begin{equation*}
         {\left\| {{{\bf A}} - {\bf B}} \right\|_{F*}} = \left\{ \begin{array}{l}
        {\left\| {{{\bf A}_s}-{{\bf B}_s}} \right\|_F} + \frac{{\left\| {{{\bf{A}}_i} - {{\bf B}_i}} \right\|_F^2}}{{2{{\left\| {{{\bf{A}}_s} - {{\bf B}_s}} \right\|}_F}}}\epsilon \qquad {\rm {if}} \;\; {{\bf A}_s} \ne {{{\bf B}_s}},\\
        {\left\| {{{\bf A}_i} - {{\bf B}_i}} \right\|_F}\epsilon \;\;\;\; \quad\quad\quad \qquad \qquad {\rm {otherwise}}.
    \end{array} \right.
    \end{equation*}
\end{definition}
\begin{theorem}[Average quasi-metric error]
\label{the:qme}
    Under the hypotheses of \Cref{the:fe}, 
    we have
   \begin{align*}
    {\bb E} {\left\| {\widehat {\bf{A}} - {\bf{A}}} \right\|_{F*}} &\le  {\left( {1 + \frac{r}{{p - 1}}} \right)^{1/2}}{\left( {\sum\nolimits_{j > r} {\sigma _j^2} } \right)^{1/2}} \\
    &\quad+  \left( \left\| {   \widehat{\overline{{\bf U}}}_s{\bf B}_i\widehat{{\bf V}}_s^* } \right\|^2_F + \left\| {\bf{Q}}_i^* \right\|_F^2 \left\| {{\bf{E}}_s} \right\|_F^2 + \left\| \overline{{\bf Q}_i^*} \right\|_F^2 \left\| {{\bf{E}}_s} \right\|_F^2 \right)^{1/2} \sqrt{\epsilon},
    \end{align*}
    where $\widehat{\overline{{\bf U}}}_s$  and $ \widehat{{\bf V}}_s^*$ are such that $\left[ {\overline{{\bf U}}}_s, \widehat{\overline{{\bf U}}}_s\right] \in {\bb C}^{m\times m}$ and $\left[ {\bf V}_s^*, \widehat{{\bf V}}_s^* \right] \in {\bb C}^{n\times n}$ are unitary and $\overline {{\bf{Q}}_i}$ is such that $\overline {{\bf{Q}}_i^*}{{\bf{U}}_s} + \overline {{\bf{Q}}_s^*} {{\bf{U}}_i} = {\bf{O}}$ with ${\overline{{\bf Q}_s}}$ being such that $\left[ {{\bf Q}_s}, {\overline{{\bf Q}_s}} \right] \in {\bb {C}}^{m\times m}$ is unitary.
    \end{theorem}
    \begin{proof}
        From the definition of quasi-metric, we have 
        $ {\left\| {\widehat {\bf{A}} - {\bf{A}}} \right\|_{F*}^2} =  \left\| { {\bf E}_s} \right\|^2_F + \left\| {  {\bf E}_i } \right\|^2_F \epsilon$. Hence, 
        we only need to obtain the bound for $\left\| {  {\bf E}_i } \right\|^2_F$. 
        
By the property of Frobenius norm, we have 
         \begin{align*}
             \left\| {  {\bf E}_i } \right\|^2_F  
             &= \left\|  \left( \begin{array}{l}
{\bf{Q}}_s^*\\
\overline {{\bf{Q}}_s^*} 
\end{array} \right) \left( {{\bf{A}}_i} - {{\bf{U}}_i}{{\bf{\Sigma }}_s}{\bf{V}}_s^* - {{\bf{U}}_s}{{\bf{\Sigma }}_s}{\bf{V}}_i^* \right) \right\|^2_F \\
&= \left\| {\begin{array}{*{20}{l}}
{{\bf{Q}}_s^*{{\bf{A}}_i} - {\bf{Q}}_s^*{{\bf{U}}_i}{{\bf{\Sigma }}_s}{\bf{V}}_s^* - {\bf{Q}}_s^*{{\bf{U}}_s}{{\bf{\Sigma }}_s}{\bf{V}}_i^*}\\
{\overline {{\bf{Q}}_s^*} {{\bf{A}}_i} - \overline {{\bf{Q}}_s^*} {{\bf{U}}_i}{{\bf{\Sigma }}_s}{\bf{V}}_s^* - \overline {{\bf{Q}}_s^*} {{\bf{U}}_s}{{\bf{\Sigma }}_s}{\bf{V}}_i^*}
\end{array}} \right\|_F^2 \\
&= \left\| {{\bf{Q}}_s^*{{\bf{A}}_i} - {\bf{Q}}_s^*{{\bf{U}}_i}{{\bf{\Sigma }}_s}{\bf{V}}_s^* - {\bf{Q}}_s^*{{\bf{U}}_s}{{\bf{\Sigma }}_s}{\bf{V}}_i^*} \right\|_F^2 \\
&\quad+ \left\| {\overline {{\bf{Q}}_s^*} {{\bf{A}}_i} - \overline {{\bf{Q}}_s^*} {{\bf{U}}_i}{{\bf{\Sigma }}_s}{\bf{V}}_s^* - \overline {{\bf{Q}}_s^*} {{\bf{U}}_s}{{\bf{\Sigma }}_s}{\bf{V}}_i^*} \right\|_F^2.
         \end{align*}
From the operations in \Cref{alg:rdsvd}, it is seen that  $ \overline {\bf{U}} _s={\bf{Q}}_s^*{{\bf{U}}_s}$, ${\bf{Q}}_i^*{{\bf{U}}_s} + {\bf{Q}}_s^*{{\bf{U}}_i} = {\overline {\bf{U}} _i}$ and ${\bf{Q}}_i^*{{\bf{A}}_s} + {\bf{Q}}_s^*{{\bf{A}}_i} = {{\bf{B}}_i}$. 
Further, by $\overline {{\bf{Q}}_s^*} {{\bf{U}}_s}{{\bf \Sigma}_s}{\bf V}_i^* = {\bf O}$ due to $ \overline {{\bf{Q}}_s^*} {{\bf{U}}_s}= {\bf O}$, we can check that
         \begin{equation*}
              {\bf O} = \overline {{\bf{Q}}_i^*} {{\bf{U}}_s} + \overline {{\bf{Q}}_s^*} {{\bf{U}}_i} = \overline {{\bf{Q}}_i^*} {{\bf{U}}_s}{{\bf \Sigma}_s}{\bf V}_s^* + \overline {{\bf{Q}}_s^*} {{\bf{U}}_i}{{\bf \Sigma}_s}{\bf V}_s^* + \overline {{\bf{Q}}_s^*} {{\bf{U}}_s}{{\bf \Sigma}_s}{\bf V}_i^* = \overline {{\bf{Q}}_i^*} {{\bf{A}}_s} + \overline {{\bf{Q}}_s^*} {{\bf{A}}_i}.
         \end{equation*}
         Thus,
         \begin{align*}
             \left\| {  {\bf E}_i } \right\|^2_F &
             = \left\| {{{\bf B}_i} - {\bf{Q}}_i^*{{\bf{A}}_s} - {{\overline {\bf{U}} }_i}{{\bf \Sigma}_s {\bf V}}_s^* + {\bf{Q}}_i^*{{\bf{U}}_s}{{\bf \Sigma}_s {\bf V}}_s^*  - {{\overline {\bf{U}} }_s}{{\bf \Sigma}_s {\bf V}}_i^*} \right\|_F^2 \\
             &\quad+  \left\| 
             {  { - \overline {{\bf{Q}}_i^*} {{\bf{A}}_s} + \overline {{\bf{Q}}_i^*} {{\bf{U}}_s}{{\bf{\Sigma }}_s}{\bf{V}}_s^*}   } 
             \right\|_F^2\\
             &\le  \left\| {  {\bf B}_i -  
             \overline{{\bf U}}_i{\bf \Sigma}_s{\bf V}_s^* - \overline{{\bf U}}_s{\bf \Sigma}_s{\bf V}_i^* } \right\|^2_F + 
             \left\| {{\bf{Q}}_i^*{{\bf{U}}_s}{{\bf \Sigma}_s {\bf V}}_s^* 
             - {\bf{Q}}_i^*{{\bf{A}}_s}} \right\|_F^2 \\
             &\quad+ 
             \left\| { \overline{{\bf Q}_i^*}{{\bf{U}}_s}{{\bf \Sigma}_s {\bf V}}_s^* - \overline{{\bf Q}_i^*}{{\bf{A}}_s} } \right\|_F^2.
         \end{align*}
         Considering that ${ {\overline{{\bf U}}_s^*}}{\overline{{\bf U}}_i} + { {\overline{{\bf U}}_i^*}}{\widehat {\overline{{\bf U}}}_s} = {{\bf O}_r}$ and ${{{{\bf V}_s^*}}}{{\bf V}_i} + {{{{\bf V}_i^*}}}{\widehat {\bf V}_s} = {{\bf O}_r}$, similar to \cite[Theorem 4.6]{wei2024singular}, we have $\left\| {  {\bf B}_i - \overline{{\bf U}}_i{\bf \Sigma}_s{\bf V}_s^* - \overline{{\bf U}}_s{\bf \Sigma}_s{\bf V}_i^* } \right\|^2_F  = \left\| {   \widehat{\overline{{\bf U}}}_s{\bf B}_i\widehat{{\bf V}}_s^* } \right\|^2_F$. Hence,
         \begin{align*}
             \left\| {  {\bf E}_i } \right\|^2_F  
             \le \left\| {   \widehat{\overline{{\bf U}}}_s{\bf B}_i\widehat{{\bf V}}_s^* } \right\|^2_F + \left\| {{\bf Q}}_i^* \right\|_F^2 \left\| {{\bf{E}}_s} \right\|_F^2 + \left\| \overline{{\bf Q}_i^*} \right\|_F^2 \left\| {{\bf{E}}_s} \right\|_F^2,
         \end{align*}
         which together with \cite[Theorem 10.5]{halko2011finding} and the Cauchy-Schwarz inequality implies the desired result. 
         \end{proof}


\section{Numerical experiments}
\label{sec:exp}
Two experiments are provided to evaluate the performance of the proposed algorithms. 
The main baseline is 
CCDSVD. We record the running time and relative errors, 
where the latter are defined by
\begin{align*}
    &\frac{{{{\left\| {{{\bf A}_s} - {{\bf U}_s}{{\bf \Sigma}_s}{\bf V}_s^*} \right\|}_F}}}{{{{\left\| {{{\bf A}_s}} \right\|}_F}}}, \quad\quad\quad\quad\quad\quad\;\; {\rm for\; the\; standard\; part}, \\
    &\frac{{{{\left\| {{{\bf{A}}_i} - {{\bf{U}}_i}{{\bf \Sigma}_s {\bf V}}_s^* - {{\bf{U}}_s}{{\bf \Sigma}_s {\bf V}}_i^*} \right\|}_F}}}{{{{\left\| {{{\bf{A}}_i}} \right\|}_F}}}, \quad\quad {\rm for\; the\; infinitesimal\; part} .\\
    &\frac{{{{\left\| {{{\bf{A}}_i} - {{\bf{U}}_i}{{\bf \Sigma}_s {\bf V}}_s^*- {{\bf{U}}_s}{{\bf \Sigma}_i {\bf V}}_s^* - {{\bf{U}}_s}{{\bf \Sigma}_s {\bf V}}_i^*} \right\|}_F}}}{{{{\left\| {{{\bf{A}}_i}} \right\|}_F}}}, \quad\quad {\rm for\; the\; infinitesimal\; part\; of \;CDSVD} .
\end{align*}
All experiments are performed 20 runs by using Matlab
2022a on a computer equipped with an Intel Core i9-12900KF 3.20GHz CPU and 128 GB memory.

In the first experiment, we consider dual matrices ${\bf A} = {\bf A}_s + {\bf A}_i\epsilon $ of size $2m\times m$, where the rank-$r$ matrices ${\bf A}_s $ and ${\bf A}_i $ are formed by multiplying ${\bf B}$ of size $2m\times r$ and ${\bf C}$ of size $r\times m$ generated randomly by the Matlab function $\textsc{randn}(\cdot)$.  The rank $r$, the oversampling parameter $p$, and the power scheme parameter $q$ are set to $m/5$, $10$, and $1$, respectively. We implement  CDSVD, CCDSVD, RCCDSVD and RCCDSVD with the random dual matrix ${\bf \Omega}$ for $m = 2500, 5000, 7500$, and $ 10000$. Note that the aforementioned ${\bf \Omega}$ will be dual real if  $B$ and $C$ are real, and  dual complex if $B$ and $C$ are complex. 

\Cref{table} shows the numerical results for this experiment, 
where RE1 and RE2 represent the relative errors for the standard part and 
the infinitesimal part, respectively. It can be seen that CCDSVD and CDSVD have similar relative errors, but the running time of the former is a little shorter than that of the latter. 
The randomized algorithms effectively reduce the computational cost with the accuracy being similar. As for the two randomized algorithms, the one with random dual matrix performs a little worse in terms of errors and running time. 

\begin{table}[htbp]
\tiny
    \centering
    \caption{Numerical results for real and complex matrices for different $m$. RCCDSVD2 denotes RCCDSVD with random dual matrix.}
    \begin{tabularx}{\textwidth}{@{}l*{4}{c}*{4}{c}@{}}
        \toprule
        & \multicolumn{4}{c}{\textbf{Real matrices}} & \multicolumn{4}{c}{\textbf{Complex matrices}} \\
        \cmidrule(lr){2-5} \cmidrule(lr){6-9}
        Method & CDSVD & CCDSVD & RCCDSVD & RCCDSVD2 & CDSVD & CCDSVD & RCCDSVD & RCCDSVD2 \\
        \midrule
        $m = 2500$ & & & & & & & & \\
        RE1 & 3.41E-15 & 3.41E-15 & 6.82E-14 & 3.57E-13 & 4.95E-15 & 4.95E-15 & 2.31E-14 & 3.67E-14 \\
        RE2 & 1.97E-14 & 2.12E-14 & 3.70E-13 & 2.18E-12 & 1.38E-14 & 3.02E-14 & 1.21E-13 & 2.67E-13 \\
        Time & 2.49 & 2.30 & 0.68 & 0.73 & 9.72 & 9.05 & 1.84 & 2.00 \\
        \midrule
        $m = 5000$ & & & & & & & & \\
        RE1 & 5.10E-15 & 5.10E-15 & 3.91E-14 & 1.27E-13 & 7.10E-15 & 7.10E-15 & 3.18E-14 & 5.56E-14 \\
        RE2 & 2.85E-14 & 3.07E-14 & 3.07E-13 & 1.33E-12 & 2.07E-14 & 4.35E-14 & 2.29E-13 & 5.73E-13 \\
        Time & 20.35 & 18.62 & 4.14 & 4.45 & 91.82 & 86.62 & 13.40 & 14.44 \\
        \midrule
        $m = 7500$ & & & & & & & & \\
        RE1 & 5.95E-15 & 5.95E-15 & 6.35E-14 & 5.46E-13 & 8.93E-15 & 8.93E-15 & 4.41E-14 & 5.24E-14 \\
        RE2 & 3.23E-14 & 3.77E-14 & 6.70E-13 & 8.93E-12 & 2.70E-14 & 5.35E-14 & 3.68E-13 & 6.81E-13 \\
        Time & 62.99 & 57.94 & 13.11 & 14.09 & 318.83 & 301.48 & 43.53 & 46.32 \\
        \midrule
        $m = 10000$ & & & & & & & & \\
        RE1 & 6.69E-15 & 6.69E-15 & 8.02E-14 & 8.11E-14 & 1.02E-14 & 1.02E-14 & 5.17E-14 & 7.23E-14 \\
        RE2 & 3.78E-14 & 4.38E-14 & 7.76E-13 & 1.17E-12 & 3.10E-14 & 6.21E-14 & 5.53E-13 & 9.71E-13 \\
        Time & 146.97 & 135.51 & 29.16 & 30.75 & 740.95 & 701.63 & 94.26 & 100.27 \\
        \bottomrule
    \end{tabularx}\label{table}
\end{table}

In the second experiment, we test the effect of 
values of $q$ 
on the approximation errors. 
Two standard images Lena and Pepper with $512 \times 512$ pixels are first set as 
$\overline{\bf A}_s \in {\bb {R}}^{512\times 512}$ and $\overline{\bf A}_i \in {\bb {R}}^{512\times 512}$, respectively. Then, we generate complex matrices through 2-dimensional discrete Fourier transform, which results in 
the test matrix ${\bf A} = {\bf A}_s + {\bf A}_i\epsilon \in {\bb {DC}}^{512\times 512}$. 
\Cref{fig:power} depicts the approximation errors for the standard part and the infinitesimal part for $r$ ranging from 5 to 200 and $p = 4$. We can see that the power method can indeed improve the accuracy of the algorithm. However, the improvement becomes unremarkable as $q$ increases.

\begin{figure*}[htbp] 
\centering 
\subfloat[Lena, the standard part]{\begin{minipage}[b]{0.5\textwidth}
    \includegraphics[width=1\textwidth]{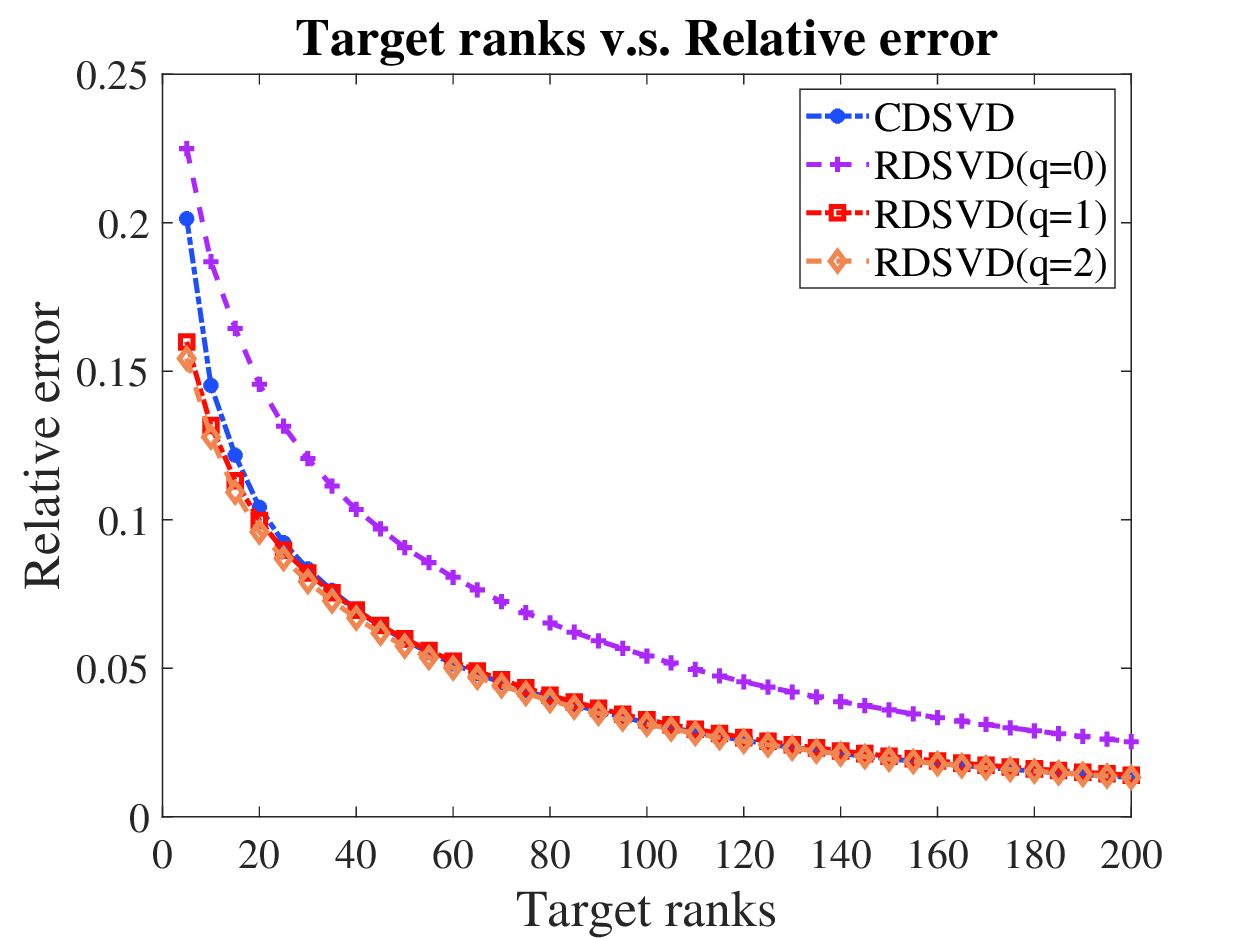} 
\end{minipage}}
\subfloat[Pepper, the infinitesimal part]{\begin{minipage}[b]{0.5\textwidth}
    \includegraphics[width=1\textwidth]{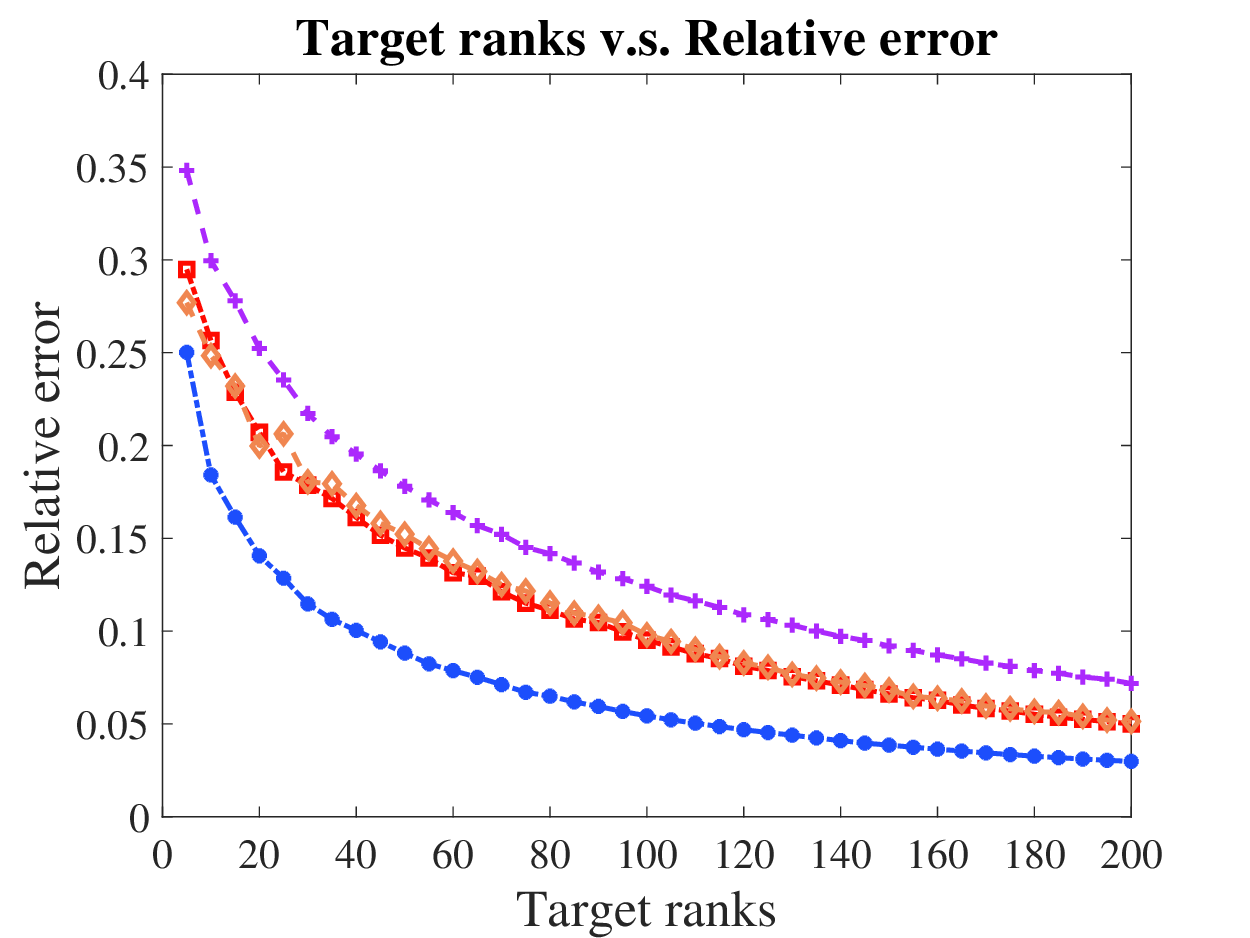} 
\end{minipage}}
 \caption{Errors incurred for different power schemes.} 
 \label{fig:power} 
\end{figure*}

\section{Concluding remarks}
\label{sec:con}
In this paper, we first simplify  CDSVD and then propose its randomized version. 
Further, the expected approximation errors of the randomized algorithm under Frobenius norm and a quasi-metric are provided. Numerical experiments demonstrate the effectiveness of our CCDSVD and the speedup of the randomized algorithm. 

Currently, 
dual quaternions are becoming increasingly popular \cite{qi2022dual,ling2022singular}. Future work can explore randomized algorithms for dual quaternions problems. 

\bibliographystyle{siamplain}
\bibliography{refffff}
\end{sloppypar}
\end{document}